\documentclass{amsart} 
\usepackage{appendix}
\usepackage{graphicx}
\usepackage[parfill]{parskip} 
\usepackage{amsfonts, amscd, mathrsfs, amsmath, amssymb, amsthm}
\usepackage{url}
\usepackage{hyperref} 
\hypersetup{backref,pdfpagemode=FullScreen,colorlinks=true}
\usepackage{tikz-cd}
\usepackage{mathtools}
\usepackage{color, verbatim}
\usepackage{bm}
\usepackage[hmargin=3cm,vmargin=3cm]{geometry}
\numberwithin{equation}{section}

\newtheorem{theorem}[equation]{Theorem} 
\newtheorem{proposition}[equation]{Proposition}

\newtheorem{claim}[equation]{Claim} 
 
\newtheorem{corollary}[equation]{Corollary}

\newtheorem{proposal}{Proposal}
\theoremstyle{definition}
\newtheorem{definition}[equation]{Definition}

\theoremstyle{remark}

\newtheorem{remark}[equation]{Remark}
\newtheorem{example}{Example}
\newtheorem{question}{Question}

\DeclareMathOperator{\Acat}{A _{\infty}Cat _{R} ^{\mathbb{Z}
_{n}, strict}}

\DeclareMathOperator{\obj}{obj}
\DeclareMathOperator{\D}{\Delta}  

\begin{document}
\title{Hamiltonian elements in algebraic K-theory}
\author{Yasha Savelyev} 
\thanks {Supported by CONAHCYT 
research grant CF-2023-I}
\address{University of Colima, faculty of science, 
Bernal Díaz del Castillo 340,
Col. Villas San Sebastian,
28045 Colima Colima, Mexico}
\begin{abstract}  
A Hamiltonian bundle $M \hookrightarrow P \to X$ (with
monotone compact fibers) induces via Floer theory
a type of ``bundle of $A _{\infty}$ categories'' over $X$,
with fiber given by the Fukaya category of $M$. Morita
theory of $A _{\infty} $ categories, the above picture for $X=S
^{m}$, and geometric representation theory yield the
following: if $G$ is a compact Lie group and $R$ is a
commutative ring then there is a natural group homomorphism
$\pi _{m} (BG) \to K ^{Cat}_{m}(R)  $, where $K
^{Cat} _{m} (R)$ are a type of categorified algebraic
$K$-theory groups of $R$, analogous to To\"en's secondary
$K$-theory. We also construct underlying maps of this type to
classical algebraic $K$-theory of $R$.  
This framework gives a geometry-powered proof that $K ^{Cat} _{2} (\mathbb{Z} )$ is infinitely
generated (with the details to appear in a future work). This is in contrast
to Quillen's finite generation result for standard algebraic
$K$-theory of $\mathbb{Z} $. 
Taking the Langlands
dual of $G$, we explore a conjectural relationship between the
images of the corresponding homomorphisms above.  
\end{abstract}
\maketitle
\section{Introduction}
A guiding idea is that a Hamiltonian bundle $M
\hookrightarrow P \to Y$  gives rise to a kind of bundle of $A _{\infty}$
categories, with fiber the Fukaya category of $M$.
Precisely what this means, is an interesting mathematical problem in its own right, and one answer is developed
in ~\cite{cite_SavelyevGlobalFukayacategoryI}. Notably, these are not local systems; as shown in ~\cite{cite_SavelyevGlobalFukayacategoryII}, the resulting bundles are generally non-trivial over spheres. 

The second idea, due to
To\"en~\cite{cite_ToenThehomotopytheoryofdg-categoriesandderivedMoritatheory}, is that there is a refinement
of algebraic $K$-theory of a commutative ring $R$, by replacing
$R$-modules with $dg$ or $A _{\infty}$ categories over $R$.
We show that these two
ideas naturally fit together. We use the above
bundle assignment, with $Y=S ^{m}$,  to associate to each
such $P$ an element in categorified and standard algebraic
$K$-theory of $R$. In particular, we get a geometric
insight into algebraic $K$-theory
of $\mathbb{Z} $, especially its categorified variant $K
^{Cat} _{m} (\mathbb{Z} )$. For example, we obtain an infinite
generation result for $K ^{Cat} _{2} (\mathbb{Z} )$. At the
end we explore connections with Langlands duality. We thus
link algebraic $K$-theory, Floer theory, and
aspects of geometric representation theory in one framework.  

Although in the main examples here the principal actors are
compact groups, and in particular $PU (n)$, the framework is
most naturally understood from the perspective of the group
of Hamiltonian symplectomorphisms of some symplectic
manifolds. We quickly introduce the latter.
Let $(M, \omega )$ be a symplectic manifold, throughout
assumed to be compact, with $\omega
$ a closed nondegenerate 2-form on $M$. The group of
diffeomorphisms $\phi $ of $M$,
satisfying $\phi ^{*} \omega = \omega  $, is called the
symplectomorphism group of $(M, \omega )$. When $M$ is
simply connected, this
is also the group of Hamiltonian symplectomorphisms, which we
denote by $\mathcal{H}  = \operatorname {Ham} (M,  \omega)
$. In general,
a symplectomorphism is Hamiltonian if it is the time-one map
of an isotopy generated by a family of vector fields $\{X
_{t}\}$, $t \in [0,1]$, satisfying $$\omega (X _{t}, \cdot) = dH _{t}
(\cdot),$$ for
some family of smooth functions $\{H _{t}\}$ on $M$.

The group $\mathcal{H} $ is one of the
more enigmatic objects of mathematics. It is always (unless
$M=pt$) an
infinite-dimensional Fr\'echet Lie group with its $C
^{\infty}$ topology. But it has a natural, bi-invariant
Finsler metric called the Hofer metric. So in some ways
it behaves like a compact group. Its topological and algebraic
structure is tied to the main problems of Hamiltonian dynamics, see for instance Polterovich
~\cite{cite_PolterovichThegeometryofthegroupofsymplecticdiffeomorphism}.

Much of the algebraic/topological structure of $\mathcal{H} $ is reflected in the
properties of Hamiltonian fibrations, which are the main
geometric ingredients in what follows. A \emph{Hamiltonian fibration} $M \hookrightarrow P \to Y$ is
a smooth fiber bundle over $Y$, with fiber a symplectic
manifold $(M,\omega)$, whose
structure group is contained in $\operatorname {Ham} (M,
\omega) $.  For example, take $(M, \omega
) = (\mathbb{CP} ^{n}, \omega _{FS})$, where $\omega _{FS}$
is the Fubini-Study symplectic form. We will from now on omit $\omega
_{FS}$ from all notation; it is implicitly always the chosen
symplectic form on $\mathbb{CP} ^{n} $.
And suppose for instance that the structure
group of $\mathbb{CP} ^{n} \hookrightarrow P \to Y$ is $G= PU (n+1) \subset \operatorname {Ham} (\mathbb{CP}
^{n}) $. 

The other ingredient is algebraic $K$-theory, which is a construction associating
an infinite loop space \footnote{In this paper, the term infinite loop space is meant to be synonymous with connective $\Omega$-spectrum, i.e. a sequence of spaces $\{E _{n}\} _{n \in
\mathbb{N}}$ with $E _{n} \simeq \Omega {E _{n+1}}$.  Then
$K (R)$ is the notation for the $0$-space $E _{0}$ of the
corresponding spectrum.} $K (R)$ to any ring $R$; and so abelian groups:
\begin{equation*}
K _{m} (R) = \pi _{m}  (K (R)), \quad m \in \mathbb{N},
\end{equation*}
called $K$-theory groups of $R$.
This is a powerful tool in
algebraic topology and ring theory, with applications for
instance in
number theory and geometric topology.

One motivating example for algebraic $K$-theory is the following.
Take $R $ to be the non-commutative ring  $\mathcal{K}$ of compact operators on a separable complex Hilbert space,  the
algebraic $K$-theory space $K (\mathcal{K})$ \footnote
{Rather one must take a certain non-connective variant of the algebraic $K$
theory construction.} is
identified with the topological $K$-theory space $KU = BU
\times \mathbb{Z} $, see Karoubi ~\cite{cite_KaroubiKtheory},
Suslin-Wodzicki ~\cite{cite_SuslinWodAlgebraicKtheory}.   
Thus a complex vector bundle $E$ over a space $X$
determines a homotopy class $[f _{E}: X \to
K (\mathcal{K})]$, which we might call a \emph{geometric
element} of the corresponding abelian group.
We shall show that in a certain sense this extends to
a general ring $R$, provided it is commutative, and
replacing complex vector bundles by Hamiltonian fibrations. 

For each commutative ring $R$, we construct abelian groups $K ^{Cat}
_{m}(R)$, called categorified
algebraic $ K$-theory groups of $R$.  These
can be understood to be generated by the aforementioned bundles of $A _{\infty}$ categories over $S ^{m}$ with fibers $ A _{\infty} $ categories with finite type, see Definition \ref{def_finitetype}. 
In principle, this
categorified algebraic $K$-theory of $R$ is only a slight
variation of the secondary $K$-theory of $R$ introduced by
To\"en, with the main difference being that $ \mathbb{Z} $-graded
chain complexes are replaced by $\mathbb{Z} _{2}$-graded
chain complexes. 
The following is
a conjunction of Theorem \ref{thm_compactmonotoneWHFP} and
Corollary \ref{cor_CPn}.
\begin{theorem} \label{thm_IntroI} 
If the Fukaya category of $ M$ is finite type, then for each
$R$ the data of the isomorphism class of $M \hookrightarrow
P \to S ^{m}$ determines
an element in $K ^{Cat} _{m} (R)$. 
This determines homomorphisms $$\phi _m: \pi _{m} (B
\operatorname {Ham} (\mathbb{CP} ^{n})) \to
K ^{Cat} _{m} (R),$$
and so homomorphisms $$\phi _m: \pi _{m} (BPU (n)) \to
K ^{Cat} _{m} (R).$$
\end{theorem}
\begin{example} \label{exm_pu2}
In the case of $n=1$, in
~\cite{cite_SavelyevGlobalFukayacategoryII}, using geometric
arguments we outline the proof of non-triviality of  
$\phi  _{4}$ for $R=\mathbb{Z}$.  This contrasts with a theorem of Rognes
\cite{cite_RognesK4Zistrivial} on the triviality of $K _{4}
(\mathbb{Z} ) $.
\end{example}
\begin{claim} $$\phi  _{2}: \mathbb{Z} _{p} \simeq \pi _{2} (B
PU (p)) \to K _{2} ^{Cat} (\mathbb{Z} )$$ is injective for
each prime $p$. Consequently, $K
_{2} ^{Cat} (\mathbb{Z} )$ is infinitely generated.
\end{claim}
Since the groups $K _{m}
(\mathbb{Z} )$ are always
finitely generated by the foundational
results of Quillen
\cite{cite_QuillenHigherAlgebraicKtheory}, the claim is perhaps surprising.

One geometric ingredient for the proof is the now
classical part of symplectic geometry: the Seidel
homomorphism
~\cite{cite_Seidelpi1ofsymplecticautomorphismgroupsandinvertiblesinquantumhomologyrings}.
The details will appear in a future work.
\subsection{Generalizing to $G/T$.} \label{sec_Generalizing to GT.}
Theorem \ref{thm_IntroI} extends to $M=G/T$, where $G$ is
compact.
However, the details in that case also involve
non-elementary geometric representation theory, see Section
\ref{sec_Generalizing to GT}. We still get natural
maps 
\begin{equation*}
  \pi _{m} (B \operatorname {Ham}
	(G/T)) \to K ^{Cat} _{m}
	(R),
\end{equation*}
but after fixing a certain choice (a distinguished toric
degeneration of $G/T$). Without that we get into questions of
finiteness of the Fukaya category $\operatorname {Fuk} (G/T) $. %

\subsection{Pushing forward to classical $K$-theory} \label{sec_Pushing forward to classical $K$-theory}
To partially bypass the need for choices, we can pass to
classical algebraic $K$-theory,
instead of its categorified analogue.  
We take a 2-periodic variant of classical algebraic
$K$-theory of $R$, which we denote by $K _{m}^{\mathbb{Z}
_{2}} (R)$. Technically, $K ^{\mathbb{Z} _{2}} (R)$ is the
algebraic $K$-theory infinite loop space $K (W)$, where $W$ is the Waldhausen category of 2-periodic complexes of $R$-modules. 

\begin{remark} \label{rem_Toen}
It was
pointed out to me by Bertrand To\"en that for $R$ a regular
ring, for example $R= \mathbb{Z} $, $K ^{\mathbb{Z} _{2}}
_{m} (R) \simeq K _{m} (R) \oplus
K _{m-1} (R)$. This can be proved with $A ^{1}$ homotopy
theory. Although we don't technically need this.
\end{remark}
The same basic framework of the paper yields the following.
\begin{theorem} \label{thm_intro}
For any compact Lie group $G$, there is
a homotopy natural map: 
\begin{equation} \label{eq_intro2}
h: \operatorname {BG} \to K ^{\mathbb{Z} _{2}}
(R),
\end{equation}
and so there are natural homomorphisms $$h _{m}: \pi _{m} (\operatorname {BG}) \to
K ^{\mathbb{Z} _{2}} _{m} (R).$$
\end{theorem}
Taking the Langlands dual of $G$, we explore correspondence
of the respective mappings in Section \ref{sec_Mirror}. 

\begin{question} \label{que_1}
Is it possible to obtain the classes $$[\operatorname {BG}
\to K ^{\mathbb{Z} _{2}}(R)] $$ by techniques of homotopy
theory and algebraic geometry? In other words, can we bypass the geometric analysis of
the construction? 
\end{question}
It should be noted that compared to the maps $\phi _{m}$,
showing non triviality of the maps $h _{m}$ is harder (and
is an open problem), even
restricting to $PU (n)$ and $R=\mathbb{Z} $. This might seem
paradoxical, but the idea is that $K ^{Cat} _{m} (R)$ sees more
information, so it is easier to find nontrivial things in
the image of $\phi _{m}$. 

%
%

\section{Preliminaries} \label{sec_Preliminaries}
We will mostly follow Keller
~\cite{cite_KellerOndifferentialgradedcategories}, but the base commutative ring will be denoted 
$R$ to avoid overloading notation. An $A
_{\infty}$ category $C$ over $R$ is a form of
non-associative dg-category, with $hom _{C}(a,b)$ a cochain
complex over $R$. We refer the reader to standard references
for further details, see for instance
\cite{cite_colimits} which will come in use later.

Denote by $A
_{\infty}\operatorname {Cat}  _{R} ^{strict}$ the category of strictly unital, $\mathbb{Z} $-graded $A _{\infty}$ categories with
morphisms being strict unital $A _{\infty} $ functors. We set
$\Acat$ to be like $A _{\infty}\operatorname {Cat}  _{R}
^{strict}$;  but with chain complexes
$\mathbb{Z} _{n}$-graded instead of $\mathbb{Z} $-graded.
This means that all structure maps, functors, etc., are with 
respect to the $\mathbb{Z} _{n}$-grading.
Various subscripts $R$ may be omitted.


There is an endofunctor $$T: \Acat \to \Acat, $$ s.t. $T (B)$ is idempotent complete and pretriangulated
for each $B$. We may denote
by $\widehat{C} $ the category $T (C)$.

Let $\operatorname {Cat}  _{R} ^{\mathbb{Z} _{n}}$ denote the category of small $R$-linear $\mathbb{Z} _{n}$-graded categories.
For $C \in \Acat$, 
set $hC \in Cat _{R} ^{\mathbb{Z} _{n}}$ to be the category with the same
objects and with $hom _{hC} (a,b) = H ^{*} (hom _{C} (a,b))$.

Given a morphism $f: A \to B$ in $\Acat$, we denote by $f _{*}: hA \to hB$ the induced morphism in $Cat _{R} ^{\mathbb{Z} _{n}}$. We say that
$f$ is a \emph{quasi-equivalence} if $f
_{*}$ is an equivalence.  
We will say that $f$ is
a \emph{Morita equivalence} if the induced
functor $T (A) \to T (B)$ is a quasi-equivalence. We denote by $Hmo ^{\mathbb{Z} _{2}}$ the
homotopy category of $A _{\infty}\operatorname {Cat}  _{R} ^{strict,
\mathbb{Z} _{2}}$ with respect to Morita equivalences, as
explained by Keller this
is a pointed category with the zero object: the
category $ 0$ with one object and zero endomorphism
complex.

The category $\Acat$ is cocomplete, see
\cite{cite_colimits}. Furthermore, it has a model category
structure \cite{cite_modelstructureAinftyCat}. Applying
left Bousfield localization along Morita equivalences we get
what is often called the Morita model structure. Although,
we do not explicitly need the language of model structures.
The main algebraic ingredient is the following.

\begin{definition}\label{def_shortexactdgsequence} A \textbf{\emph{short
exact sequence}} in $A _{\infty}\operatorname {Cat}  _{R} ^{strict, \mathbb{Z}
_{2}}$ is a sequence $A \xrightarrow{I} B 
\xrightarrow{P} C$, such that in $Hmo ^{\mathbb{Z} _{2}}
$, $I$ is the kernel of $P$ and $P$ is the cokernel of $I$. 
\end{definition} 
As explained by Keller in
~\cite{cite_KellerOndifferentialgradedcategories}, an
explicit example for such a short
exact sequence is given by a Drindfeld dg-quotient
sequence $A \hookrightarrow B \to B/A$.

\subsection{Simplicial sets} \label{sec_Simplicial sets}
We will need basic theory of simplicial sets, particularly
the notion of a simplex category, for the definition
of bundles of $A _{\infty} $ categories.
We denote by $\D$ the simplex category:
\begin{itemize}
	\item The set of objects of $\Delta^{} $ is $\mathbb{N}$.
	\item $\hom _{\D}
 (n, m) $ is the set of non-decreasing maps $[n] \to [m]$, where $[n]= \{0,1, \ldots, n\}$, with its natural order.
\end{itemize} 
A simplicial set $X$ is a functor                                 
\begin{equation*}
   X: \D ^{op}  \to Set.
\end{equation*} 
$X (n)$ is the set of $n$-simplices of $X$.
$\D ^{d}  $ will denote a particular
simplicial set: the standard representable
$d$-simplex, with $$\D ^{d} (n) = hom _{\D} (n, d).                                         
$$ 
Morphisms of simplicial sets are given by natural
transformations of functors and this forms a category
denoted by $sSet$.

An important ingredient in our construction will be the
following.  
\begin{definition} \label{definition_simplexcategory} For $X$ a simplicial set, $\Delta (X)$ will denote the \textbf{\emph{simplex category of $X$}}. This is the
category such that: 
\begin{itemize}
	\item The set of objects $\obj \Delta (X) $  is the set
of simplicial maps:
$$\Sigma: \Delta ^{d} \to X, \quad d \geq
   0. $$
	 \item Morphisms $f: \Sigma _{1} \to \Sigma
      _{2}$ are commutative diagrams in $sSet$:
\begin{equation} \label{eq:overmorphisms1}  
\begin{tikzcd}
\D ^{d} \ar[r, ""] \ar [dr,
   "\Sigma _{1}"] &  \D ^{n} \ar
   [d,"\Sigma _{2}"] \\
    & X.
\end{tikzcd}
\end{equation}
\end{itemize}
\end{definition}

Given a morphism $f: X \to Y$ of simplicial sets, we set
$ \Delta^{} f: \Delta^{} (X) \to \Delta^{}(Y)  $ to be
the functor defined on objects by     
$\Delta f (\Sigma) = f \circ \Sigma $. If $m: \Sigma _{1} \to \Sigma _{2}$ is a morphism in $\Delta (X) $:
\begin{equation*}
\begin{tikzcd}
\Delta^{k} \ar[r,
   "{m}"] \ar [dr,
   "\Sigma_1"] & \Delta^{d}  \ar
   [d,"\Sigma _{2}"] \\
   &  X,
\end{tikzcd}
\end{equation*}
then obviously the diagram below also
commutes:
\begin{equation*} 
\begin{tikzcd}
\Delta^{k} \ar[r,
   "{m}"] \ar [dr,
   "h_1"] & \Delta^{d}  \ar
   [d,"h _{2}"] \\
   &  Y,
\end{tikzcd}
\end{equation*}
where $h _{i}= \Delta f (\Sigma _{i}) = f \circ \Sigma
_{i}$,
$i=1,2$.   And so the latter diagram determines a morphism $\Delta f (m):  h _{1} \to h _{2}  $ in $\Delta
(Y) $.  Clearly, this
determines a functor $\Delta f$ as needed.

\begin{definition}\label{definition_concordaceoffunctors}
Let $X$ be a simplicial set, and $C$ any category. Let $F_i:
\Delta^{} (X) \to C$, $i=0,1$	be functors.
A \textbf{\emph{concordance}} of $F _{0}$ to $F _{1}$  is
a functor $$\widetilde{F}: \Delta^{} (X \times \Delta^{1})
\to C$$ such
that $\widetilde{F}| _{\Delta^{} (X \times \{0\})} = F _{0} $ and $\widetilde{F}| _{\Delta^{}
(X \times \{1\})} = F _{1} $, where $\{0\}, \{1\}$ here
denote the images of the natural end point inclusions
$\Delta^{0} \to \Delta^{1} $.
\end{definition}


\section{Definition of the categorified algebraic K-theory
groups of $R$} \label{sec_definition}
We need to impose some algebraic finiteness conditions on our
${A} _{\infty}$ categories. otherwise, any $K$-theory
we construct would be trivial due to the Eilenberg swindle;
the same reason that $K$-theory of infinite-dimensional
vector bundles is trivial.
A natural requirement is that our ${A} _{\infty}$
categories be smooth and proper, or saturated in the sense of
Kontsevich-Soibelman ~\cite{cite_KontsevichHomologicalSmooth}, as this is a dualizability
condition, see ~\cite{cite_ToenLectures}. In the $\mathbb{Z}
_{n}$-graded setting, the analogue of this condition is more
complex, we will instead work with
a less restrictive notion, which is not too hard
to check for geometric examples.
%

Set $\mathcal{A} = \mathcal{A} _{R} \subset {A} _{\infty}Cat
_{R}
^{\mathbb{Z} _{2}, strict}$ to be the
full subcategory  whose objects are finite type, where the latter is defined as follows. 
\begin{definition} \label{def_finitetype}
We say that $B \in {A} _{\infty}\operatorname {Cat}  _{R}
^{\mathbb{Z} _{2}, strict} $ is \textbf{\emph{finite type}}
if the following conditions are satisfied. 
\begin{itemize}
	\item $B$ is Morita equivalent to an $A _{\infty}$ algebra, whose underlying chain complex is
projective and finitely generated over $R$ in each degree.
\item In the category of 2-periodic complexes
over $R$, the Hochschild homology complex of $B$ over $R$ is
quasi-isomorphic to a complex that is
projective and finitely generated over $R$ in each degree.
\end{itemize}
\end{definition}
Set $\widehat{\mathcal{A}} _{R} \subset \mathcal{A} _{R}$ to
be the full subcategory with objects pretriangulated and
idempotent complete. As outlined by To\"en in
~\cite{cite_ToenThehomotopytheoryofdg-categoriesandderivedMoritatheory},
$\widehat{\mathcal{A}} _{R} \subset \mathcal{A} _{R}$ has
the structure of a Waldhausen category. 
The cofibrations are taken to be fully-faithful,
strict, unital, injective on sets of objects $A _{\infty}$
functors.  Weak equivalences are Morita equivalences.
To be strict, this is not a Waldhausen
category  but $\infty$-Waldhausen in the sense of Barwick ~\cite{cite_BarwickInfinityWaldhausen}, as for example the zero category $0$ (the
category with one object and zero complex as endomorphism
set) is not initial in $\widehat {\mathcal{A}} _{R}$  but the hom-sets $hom
_{\widehat {\mathcal{A}} _{R}} (0, A)$ are contractible
when taken in the $ \infty $-localization of $ \widehat
{\mathcal{A}} _{R}$ with respect to the Morita equivalences. 

Given this we define $$K ^{Cat} _{m} (R) = \pi _{m}
K (\widehat{\mathcal{A}} _{R}),$$ where $K
(\widehat{\mathcal{A}} _{R})$ is the associated Waldhausen
spectrum. Calculations from this general perspective are 
difficult, we now describe a more concrete presentation. 
\subsection{Concrete presentation of $K ^{Cat} _{m} (R)$} \label{sec_Concrete presentation of $K ^{Cat} _{m} (R)$}
%
%
%
%
%
\begin{definition}\label{def_bundleAinfty}
Let $X$ be a simplicial set. A \textbf{\emph{bundle of
$A _{\infty} $ categories over $X$}} is a functor
$F: \Delta^{}  (X) \to w \widehat{\mathcal{A}} _{R}$. $F$ will be called a \textbf{\emph{bundle functor over $X$}}.
\end{definition}

Suppose we have bundle functors over $X$,
fitting into a sequence of natural transformations of
functors to $ \widehat{\mathcal{A}} _{R}$ (as opposed to $w
\widehat{\mathcal{A}} _{R}$):
\begin{equation} \label{eq_shortexactseqFunctors}
F _{0} \to  F _{1} \to F _{2},
\end{equation}  
such that for each object $\Sigma \in \Delta^{}  (S ^{k})$
the resulting sequence in $\mathcal{A} _{R}$ is a short
exact sequence. Then we call $F _{0} \to F _{1} \to F _{2}$
a \textbf{\emph{short exact sequence}}.

\begin{definition}\label{def_}
Let $X$ be a Kan complex, 
Set $B (X, R)$ to be the free abelian group generated by the set
of concordance classes $[F]$ of functors $F: \Delta^{}  (X
) \to w\mathcal{A} _{R}$.  And we define $K ^{Cat} (X, R)$
to be the quotient of $B (X,R)$ by the subgroup generated by
\begin{equation*}
\{ [F _{0}] - [F _{1}] + [F _{2}] \,|\,   \text{there is
a short exact sequence $F _{0} \to F _{1} \to F _{2}$} \}. 
\end{equation*}
\end{definition}
This is clearly contravariantly functorial in $ X$ so if $f:
X \to Y$ is a simplicial map then we get a group
homomorphism $f ^{*}: K ^{Cat} (Y, R) \to K ^{Cat} (X, R)
$ defined on generators by $ f ^{*} ([F]) = [F \circ \Delta^{}f
]$. 

The groups $K ^{Cat} (X, R)$  are covariantly functorial in
$R$ by ``extension of scalars''.
Furthermore, $X \mapsto K ^{Cat} (X, R)$ gives a contravariant homotopy functor, as by
definition of concordance it is clear that if $f: X \to Y
$ is homotopic to $ g$ then $f ^{*} = g ^{*} $. In
particular if $X _{0}$ is homotopy equivalent to $ X _{1}$
then $K ^{Cat} (X _{0}, R) = K ^{Cat} (X _{1}, R) $.

Let $S ^{n} _{\bullet }$ be the singular complex of $S
^{n}$. 
\begin{definition}\label{def_} 
For $ k>0$, we define $$ K ^{Cat} _{n} (R) = \ker \{(
\Delta^{0}  \to S ^{n}
_{\bullet}) ^{*}: K ^{Cat} (S ^{n}
_{\bullet}, R) \to K ^{Cat}
(\Delta^{0},
R)\}, $$  
and define $$ K ^{Cat} _{0} (R) = K ^{Cat} (\Delta^{0}, R).$$
In other words, $K ^{Cat} _{0} (R)$ is the quotient of the
free abelian group generated by Morita equivalence classes $[C]$, by
the subgroup generated by:
\begin{equation*}
\{ [C _{0}] - [C _{1}] + [C _{2}] \,|\,   \text{there is
a short exact sequence $C _{0} \to C _{1} \to C _{2}$} \}. 
\end{equation*}
\end{definition}

To see relationship between the pair of definitions of $K
^{Cat} _{n} (R)$,
note that for $X$ a Kan complex a bundle functor $F: |\Delta^{}  (X)| \to w\mathcal{A} _{R}$
induces a sequence 
\begin{equation}\label{eq_inducedsequence}
   |X| \simeq |\Delta^{}  (X)|
\xrightarrow{|F|} |w \mathcal{A} _{R}| \xrightarrow{\frak i}
K ^{Cat} (R),
\end{equation}
where $|X|$ also denotes the geometric realization of $X$.
And clearly a concordance between $F _{0}, F _{1}$ induces
a homotopy between $\mathfrak i \circ F _{i}$. 
Thus we get an induced map:
\begin{equation}\label{eq_J}
\widetilde{\mathcal{J}} _{X}: B (X, R) \to Maps(|X|, K
^{Cat} (R)),
\end{equation}
where $hMaps(|X|, K ^{Cat} (R))$ is the set of
free homotopy classes $[|X| \to K ^{Cat} (R) ]$. As $K
^{Cat} (R)$ is an infinite loop space, $hMaps(|X|, K ^{Cat} (R))$ is naturally an abelian group. 
\begin{claim} The map $\widetilde{\mathcal{J}} _{X} $ passes to
the quotient, inducing a map 
\begin{equation}
{\mathcal{J}} _{X}: K ^{Cat} (X, R) \to hMaps(|X|, K
^{Cat} (R)).
\end{equation}
Furthermore, $\mathcal{J} _{X} $ is a group isomorphism, and 
induces a natural isomorphism $\mathcal{J} $ from the $Ab$-valued
functor $X \mapsto K ^{Cat} (X, R) $ to the functor $X
\mapsto Maps(|X|, K ^{Cat} (R)) $.

\end{claim}
The proof of the claim is postponed, as it involves various
algebraic topology details. Note that in general one cannot always concretely present
Waldhausen $K$-theory groups in this manner. 
The proof in our setting hinges on the fact that the  $\infty
$-localization of $\widehat{\mathcal{A}} _{R}$ with respect to Morita equivalences  
has a concrete model $\mathcal{M} $, such that concordance
classes of functors $F: \Delta (X) \to
w \widehat{\mathcal{A}} _{R}  $ are in correspondence with homotopy classes $[X \to Core
(\mathcal{M})]$, where $Core (\mathcal{M} )$ is the maximal
Kan subcomplex. As remarked in the paragraph following
Definition \ref{def_finitetype}, this $\mathcal{M} $  is the $\infty $-Waldhausen category from which the spectrum $K ^{Cat}
(R)$ is constructed.

Given the claim, there is clearly an induced map 
\begin{equation}\label{eq_wtJ}
 {\mathcal{J}} _{m}: K ^{Cat} _{m} (R) \to
\pi _{m} (K ^{Cat} (R)),
\end{equation}
and it is a group isomorphism. 

\subsection{A very basic computation} \label{sec_Relation with Toen's secondary K-theory}

%
%

We recall the definition of $K _{0} (R)$. This is
the quotient of the free abelian
group generated by the set of isomorphism classes of
projective finitely generated $R$-modules, by the
subgroup generated by
\begin{equation*}
\{[A] - [B] + [C] \,|\,   \text{exists an exact sequence $0
\to A \to B \to C \to 0$}. \} 
\end{equation*}
If $\mathcal{E} $ is a triangulated category,  
one defines $K _{0} (\mathcal{E} )$  to be the quotient of
the free abelian group, generated by the set of isomorphism
classes in $\mathcal{E} $, by the
subgroup generated by
\begin{equation*}
\{ [A] - [B] + [C] \,|\,   \text{exists a distinguished
triangle $ A \to B \to C \to A[1]$} \}.
\end{equation*}
It is known that $K
_{k} (D ^{perf} (Mod _{R}))  \simeq K _{k} (R)$,
where 
$Mod _{R}$ denotes the category of projective finitely
generated $R$-modules and 
$D ^{perf} (Mod _{R})$
denotes the derived category of perfect complexes.


If $D _{n} (Mod _{R})$ denotes the $n$-periodic derived
category,
we no longer have that
$K _{0} (R) \simeq K _{0} (D _{n} (Mod _{R}))$, as
there are certain corrections. However we have the
following.
\begin{theorem} [Saito~\cite{cite_SaitoPeriodicDerived}] \label{theorem_Saito}
For $n$ even, the natural homomorphism:
\begin{align} \label{eq_psi}
& \psi: K _{0} (D _{n} (Mod _{R})) \to K _{0} (R) \\
& \psi ([V]) = \sum_{i = 0}^{i=n}	(-1) ^{i}[H ^{i} (V)],
\end{align}
is an isomorphism.
\end{theorem}
There is a group homomorphism $$\phi: K ^{Cat} _{0} (R)
\to  K _{0} (D _{2} (Mod _{R}))$$ induced by $[C] \mapsto [Hoch _{\bullet} (C)]$ for $ Hoch _{\bullet} (C)$ the
Hochschild homology complex of $ C$. This readily follows
from the fact that the Hochschild functor $ Hoch _{\bullet}$ is
``localizing'' which means in this case that a short exact
sequence $$C _{0} \to C _{1} \to C _{2}$$ in $\mathcal{A}
_{R}$ is taken to an exact triangle $$Hoch
_{\bullet} (C _{0})] \to Hoch _{\bullet} (C _{1}) \to Hoch
_{\bullet} (C _{2}) \to Hoch _{\bullet} (C _{0}) [1],  $$ in
the derived category $D _{2} (Mod _{R})$, see Keller ~\cite[Theorem
5.2]{cite_KellerOndifferentialgradedcategories}.

\begin{corollary} \label{lemma_naturalmapKperiod}
The composition map
\begin{align*}
K ^{Cat} _{0} (\mathbb{Z})
\xrightarrow{\phi} K _{0} (D
_{2} (Mod _{\mathbb{Z} })) \xrightarrow{\psi} K _{0}
(\mathbb{Z}),
\end{align*}
is non-trivial.
\end{corollary}
\begin{proof} [Proof]
Take $C$ to be a finite type $\mathbb{Z} _{2}$-graded ${A}
_{\infty}$ algebra over $\mathbb{Z} $ such that
$HH_0 (C) = \mathbb{Z} $, $HH _{1} (C) = 0$, $HH_2 (C)
= \mathbb{Z}$, where $HH _{\bullet } (C)$ is the homology of
$Hoch _{\bullet } (C)$. As a geometric example take $C= FH _{\bullet} (L _{0}, L _{0}) $ where the
latter is the Floer cochain algebra of the equator $L _{0}$ in $S
^{2}$, then the Hochschild homology $HH _{\bullet } (C)$ is 
isomorphic to the singular homology $ H _{\bullet} (S ^{2}, \mathbb{Z}
)$, ~\cite{cite_NickSheridanOntheFukayaCategory}. 

So we get $$\psi ([Hoch _{\bullet } (C)]) = 2 [\mathbb{Z}] \neq 0 \in
K _{0} (\mathbb{Z}) \simeq \mathbb{Z}.$$
\end{proof}
A further corollary of the proof above is:
\begin{corollary} \label{cor_}
$[{\operatorname {Fuk} (S ^{2},  \omega)}]
$ represents a non-trivial element in $K ^{Cat} _{0} (\mathbb{Z})$.
\end{corollary}
One natural question is how to produce non-trivial elements
in $K ^{Cat} _{m} (R)$ for higher $m$. We will construct
such elements using Hamiltonian fibrations and Floer theory.

\section{Bundles of $ A _{\infty} $ categories from Hamiltonian fibrations} \label{sec_From geometry to fibered A categories}
The main construction of
~\cite{cite_SavelyevGlobalFukayacategoryI}  has input 
a smooth Hamiltonian fibration $M \hookrightarrow P \to Y$,
with $(M, \omega )$ a compact, monotone symplectic manifold, e.g.
$\mathbb{CP} ^{n}$. More generally, a monotone symplectic
manifold $(M, \omega ) $ is one satisfying $[\omega] = const
\cdot c _{1} (TM)$, with $const>0$,  as functionals $\pi
_{2} (M) \to \mathbb{R} ^{} $. 

Denote by $ Y _{\bullet
}$ the simplicial set of smooth maps $\Delta^{n} \to Y$ for $ \Delta^{n} $ also denoting the
standard topological $ n$-simplex.

The output is a functor
\begin{equation} \label{equation_F}
F _{P}: \Delta^{} (Y _{\bullet}) \to  w{A} _{\infty}Cat
^{\mathbb{Z} _{2}, strict} _{R}, \footnote {In ~\cite{cite_SavelyevGlobalFukayacategoryI} we
work with $R=\mathbb{Q} $ but this was only used for
inverting quasi equivalences.  As we don't need to do this
here, the construction works over
any commutative ring.}
\end{equation}
that is a bundle functor over $Y _{\bullet }$, whose concordance class $ F _{P}$ is independent of
auxiliary choices (these mainly consist of choices of almost
complex structures and Hamiltonian connections.)
For $x: \Delta^{0}  \to Y$  
a point map $F _{P} (x) = \operatorname {Fuk} (P _{x},
\omega _{x}) $, $P _{x}$ denoting the fiber over $x$. 

We can outline this further: if $\Sigma: \Delta^{1}
\to Y$ is a path from $x$ to $y$, then $F _{P} (\Sigma )$
has the set of objects $$\operatorname {obj}  F _{P} (x) \sqcup
\operatorname {obj}F _{P} (y). $$ For $L _{x} \in
\operatorname {obj}  F _{P} (x)$ and $L _{y} \in
\operatorname {obj}  F _{P} (y)$, $hom _{F _{P} (\Sigma )}
(L _{x}, L _{y})$ is a certain chain complex generated by
flat sections of an auxiliary Hamiltonian connection on
$\Sigma ^{*}P$, with end points of the sections on $L _{x}$
and $L _{y}$. \footnote {These Hamiltonian connections are
fixed per simplex, but need to satisfy compatibility
conditions involving face maps of simplices.}     

More generally, for a smooth $\Sigma: \Delta ^{n} \to Y$
denote by $v _{i}: \Delta^{0} \to Y$, $0 \leq i \leq n$,  the corresponding
vertex maps. Then $F _{P} (\Sigma )$ has the set of objects 
\begin{equation} \label{eq_objectsets}
 \sqcup _{i} \operatorname {obj}  F _{P (\Sigma)} (v _{i}).
\end{equation}
The hom sets are defined analogously so that an edge
inclusion morphism
$\Sigma ^{1}  \to \Sigma  $ in $\Delta^{}  (Y _{\bullet })$  induces a fully-faithful
embedding $F _{P} (\Sigma _{1}) \to F _{P} (\Sigma )$.
Here by an edge inclusion we mean a morphism of the type:
\begin{equation} \label{eq:overmorphisms1}  
\begin{tikzcd}
\D ^{1} \ar[r, ""] \ar [dr,
   "\Sigma ^{1}"] &  \D ^{n} \ar
   [d,"\Sigma _{2}"] \\
    & X.
\end{tikzcd}
\end{equation}

The only deep part of the construction is then the
definition of the $A _{\infty} $ structure on $F _{P}
(\Sigma )$. This involves some novel geometric topology of the
universal curve, but we will not attempt to outline this
further, instead referring the reader to the original
construction in ~\cite{cite_SavelyevGlobalFukayacategoryI}. 


To get the algebraic finiteness condition on the Fukaya
category we may impose the following condition. 
\begin{definition}\label{definition_finitetypemanifold}
We call a compact monotone symplectic manifold
of \textbf{\emph{finite monotone type}}, if its
monotone Fukaya category (see
~\cite{cite_NickSheridanOntheFukayaCategory} and also
~\cite{cite_SavelyevGlobalFukayacategoryI}) is finite
type.
\end{definition}
Let us set $\widehat{F} _{P} = T \circ F _{P} $. 
We then have the following corollary of 
~\cite[Theorem
6.6]{cite_SavelyevGlobalFukayacategoryI}. 
\begin{theorem} \label{thm_compactmonotoneWHFP} Let $M \hookrightarrow
P \to Y$ be a Hamiltonian fibration with fiber $(M, \omega
)$ finite monotone type. 
Then the assignment
\begin{equation} \label{equation_}
[P] \mapsto  [\widetilde{\mathcal{J}} _{Y _{\bullet}} (\widehat{F}
_{P})] \in hMaps (Y, K ^{Cat} (R)),
\end{equation}
is well defined, where $[P]$ is the Hamiltonian isomorphism
class of the fibration $P$, and $\widetilde{\mathcal{J} }$
is as in \eqref{eq_J}.
\end{theorem}
As an example, we have:
\begin{proposition} \label{prp_finitetype}
$\mathbb{CP} ^{n}$ is
finite monotone type.
\end{proposition}
\begin{proof} [Proof]
Specifically, we take a concrete model of $\operatorname {Fuk}
(\mathbb{CP} ^{n}) $ with the set of objects given by
the orbit of the Clifford torus $L _{0}$  by the natural
$\operatorname {Ham} (\mathbb{CP} ^{n},  \omega) $
action. In this case, it is automatic that $\operatorname {Fuk}
(\mathbb{CP} ^{n}) $ is quasi-isomorphic to
the $A _{\infty }$ algebra $hom _{\operatorname {Fuk}
(\mathbb{CP} ^{n}) } (L _{0}, L _{0})$. The
Hochschild chain complex of $hom _{\operatorname {Fuk}
(\mathbb{CP} ^{n}) } (L _{0}, L _{0})$  is quasi-isomorphic
to the Hamiltonian Floer chain complex, see for instance
~\cite{cite_NickSheridanOntheFukayaCategory}, which in
each degree is finitely and freely generated over $\mathbb{Z}
$. And so $\operatorname {Fuk}
(\mathbb{CP} ^{n}) $  is finite type for any
$R$. 
\end{proof}
In particular, taking
$Y = S ^{m}$ we get: 
\begin{corollary} \label{cor_CPn}
For each commutative ring $ R$, there are natural
homomorphisms:
\begin{equation} \label{equation_}
\phi _{m}: \pi _{m} (B \operatorname {Ham} (\mathbb{CP}
^{n}))
\to K ^{Cat} _{m} (R).
\end{equation}
And in particular natural homomorphisms:
\begin{equation} \label{equation_}
\phi  _{m}: \pi _{m} (B PU (n)) \to K ^{Cat} _{m} (R).
\end{equation}

\end{corollary}
\begin{proof} [Proof]
Let $\mathbb{CP} ^{n} \hookrightarrow  P ^{U} \to B \operatorname
{Ham} (\mathbb{CP} ^{n})$ be the 
associated universal bundle. \footnote{In case the
reader is worried about smoothness of this $P$ (which was an
assumption for the bundle functor construction), this isn't an issue as
~\cite{cite_SavelyevGlobalFukayacategoryI} treats it as
generalized smooth bundle, for which the construction still
goes through.} Using
Proposition \ref{prp_finitetype} together with Theorem
\ref{thm_compactmonotoneWHFP} we get $\phi _{m}$ as set maps.
To see that they are homomorphisms note that we have
a sequence of maps well defined up to homotopy:
\begin{equation*}
  B \operatorname {Ham} (\mathbb{CP} ^{n})
	\xrightarrow{|\widehat{F} _{P ^{U}}|} |w \widehat{\mathcal{A}} _{R}|
	\xrightarrow{\mathfrak i} K ^{Cat} (R).
\end{equation*}
By the naturality of the assignment $[P] \mapsto [F _{P}]$,
see ~\cite[Section 9.2]{cite_SavelyevGlobalFukayacategoryI}, $\phi _{m}$ are the induced
maps $$(\mathfrak i \circ |F _{P ^{U}}|) _{*}: \pi _{m} (B\operatorname {Ham} (\mathbb{CP}
^{n})) \to \pi _{m} (K ^{Cat} (R)).$$ And so are group
homomorphisms.

\end{proof}
\subsection{Generalizing to $G/T$} \label{sec_Generalizing to GT}
Given a compact connected Lie group $G$, the maximal coadjoint orbit
$G/T$ can be made monotone with respect to
the Kirillov-Kostant-Souriau symplectic structure $\omega
_{KKS}$, which for $G/T$ is uniquely determined up to deformation equivalence, see
for instance ~\cite{cite_KirilovlecturesOrbit}. We will from
now on omit this monotone $\omega _{KKS}$ from notation, but
implicitly it will always be the chosen symplectic form. 
If $\operatorname {Fuk} (G/T)$ was known to be of finite
type, we would similarly get homomorphisms $\pi _{m} (B\operatorname {Ham} (G/T)) \to
K ^{Cat} _{m} (R)$. But even in this case without an
explicit model for the Fukaya category--like in the case of
$\mathbb{CP} ^{n}$--these morphisms would be far from
concrete. We address both issues as follows.

If $G=PU (n)$ there is a distinguished monotone
Lagrangian torus $L _{GC}$ of $PU(n)/T$, arising as the
central fiber of a Gelfand-Cetlin integrable system, see
~\cite[Theorem B]{cite_MonotoneLagrangiansChoKim2021}. In
a way this extends.

For a general compact $G$, $G/T$ admits a flat toric degeneration to a toric variety
$X(\Delta)$ associated with a string polytope $\Delta$, see ~\cite{cite_Caldero2002}, ~\cite{cite_AlexeevBrion2004}. 
This deformation defines a completely integrable system on
$G/T$, with moment map to $\Delta^{} $   ~\cite{cite_ToricDegenHaradaKaveh2015}.
By anchoring the construction to the standard string
valuation, one identifies a distinguished toric degeneration where the resulting string polytope is Fano. Within this fixed
framework, the fiber over the barycenter of $\Delta$ is
a monotone Lagrangian torus that is well-defined up to
Hamiltonian isotopy ~\cite{cite_Cho2025}. This torus will be
denoted as $\mathcal{L}$. (Not that taking ``non-standard''
valuations, we may get exotic tori that are not Hamiltonian
isotopic to $\mathcal{L} $ ~\cite{cite_Cho2025}.)

Thus we get
a finite type full subcategory $C _{\mathcal{L}} \subset
\operatorname {Fuk} (G/T) $, with the set of
objects the orbit of $\mathcal{L}$ by the
$\operatorname{Ham}(G/T) $ action \footnote{$C _{\mathcal{L}
}$ is possibly not known to be Morita equivalent to
$\operatorname {Fuk} 
(G/T)$, so it is just our model in the general sense.}. 
We then construct a
functor
\begin{equation*}
 F _{P ^{U}}: \Delta^{}(B \operatorname {Ham} (G/T) _{\bullet}) \to
 w \mathcal{A} _{R}
\end{equation*}
as follows. Denote by $G/T \hookrightarrow P ^{U} \to
B \operatorname {Ham} (G/T)$ the associated universal bundle. 
For a point map $x: \Delta^{0} \to X$ we set $F _{P ^{U}} (x)$ to be the full
subcategory of $\operatorname {Fuk}  (P _{x})$, with the set of objects $\phi
_{x}(\obj C _{\mathcal{L}})$, where $\phi _{x}: G/T \to P$
is a frame over $x$ (a parametrization respecting the structure
group). The remaining construction of $F _{P ^{U}}$ is
the same as the general construction of \eqref{equation_F}. 

Given this $F _{P _{U}}$, following the proof of Corollary
\ref{cor_CPn},  we again get maps:
\begin{equation} \label{equation_phiL}
\phi _{\mathcal{L}, G}: \pi _{m} (B \operatorname {Ham}
(G/T)) \to K ^{Cat} _{m} (R).
\end{equation}
Composing with the natural map $BG \to B \operatorname {Ham}
(G/T) $ we also get:
\begin{equation} \label{equation_phiL}
\phi _{\mathcal{L}, G}: \pi _{m} (BG) \to K ^{Cat} _{m} (R),
\end{equation}
and these later maps are in essence concrete, at least
within the framework of the representation theory outlined
above. 
\begin{remark} \label{rem_}
By the last sentence we mean that the data does not hide
anything transcendental unless one counts Hamiltonian
perturbations underlying Floer theory (which is used in the
construction of $F _{\mathcal{L}}$).  But since the
perturbations can be made concrete, this is not an issue. 
\end{remark}


\section{Hamiltonian elements in classical algebraic
$K$-theory of $R$} \label{sec_From categorified algebraic $K$-theory to
classical algebraic $K$-theory}
One of the goals is to obtain Hamiltonian
elements in classical algebraic $K$-theory of $R$.  
Let $ M \hookrightarrow P \to Y$ be a Hamiltonian fibration,
with compact monotone fiber $(M, \omega) $. 
Let 
$$F _{P}: \Delta^{} (Y _{\bullet}) \to  w{A} _{\infty}Cat
^{\mathbb{Z} _{2}, strict} _{R}$$
be as in the previous section. Then the composition:
$\Phi _{P} = Hoch _{\bullet} \circ  F _{P}$ gives a functor to $ w Ch
_{R} ^{\mathbb{Z} _{2}}$, where $ Ch
_{R} ^{\mathbb{Z} _{2}}$ is the Waldhausen (strict) category of 2-periodic complexes of
projective finite rank $R$-modules. The Waldhausen
$K$-theory infinite loop space $K (Ch
_{R} ^{\mathbb{Z} _{2}})$ will be denoted by $K
^{\mathbb{Z} _{2}} (R)$.

Restrict now to $Y = B \operatorname {Ham}
(G/T)$, and $P$ the associated universal $G/T$ bundle
over $\operatorname {BHam} (G/T) $. 
Now by standard homotopy theory, $ |\Delta^{} (Y _{\bullet})|
\simeq X$, where for a category $ C$, $ |C|$ is shorthand for
the geometric realization of the nerve.
Then we get a map:
\begin{equation*}
 |F _{P}|:  B \operatorname {Ham} (G/T) \to |w Ch
_{R} ^{\mathbb{Z} _{2}}|,
\end{equation*}
whose homotopy class is well defined by the fact that the
concordance class of $ F _{P}$ is well defined. 

There is a natural map $|w Ch
_{R} ^{\mathbb{Z} _{2}}| \to K ^{\mathbb{Z}  _{2}} (R)$,
by  generalities of the Waldhausen construction, see
~\cite[Remark 8.3.2]{cite_WeibelKbook}. Composing with this
map we get a 
natural homotopy class:
\begin{equation} \label{eq_mainGT}
  [|F _{P}|:  B \operatorname {Ham} (G/T) \to  K ^{\mathbb{Z}  _{2}} ({R})].
\end{equation}
\begin{proof} [Proof of Theorem \ref{thm_intro}]
Compose the natural map 
$$BG \to \operatorname {BHam} (G/T), $$ and the map from
\eqref{eq_mainGT} to obtain $h$.
\end{proof}
Of course, at the moment $h$  is very transcendental, unless
we give
$\operatorname {Fuk} (G/T) $ a concrete model as we
do in Section \ref{sec_Generalizing to GT}. Or unless we
reinterpret the map \eqref{eq_mainGT} in some more basic
terms, as asked in Question \ref{que_1}.

\section{Langlands dual algebraic $K$-theory elements.}
\label{sec_Mirror} It is natural to ask how Langlands
duality fits into our story. 
Denote by $G ^{\vee}$ the maximal compact subgroup of
the Langlands dual $G ^{\vee} _{\mathbb{C}}$ of the complexification $G
_{\mathbb{C}}$. Via Theorem \ref{thm_intro}, we have natural maps:
\begin{align*} 
h _{G}: \pi _{m} (BG) \to K ^{\mathbb{Z} _{2}}
_{m} (R), \\
h _{G ^{\vee}}: \pi _{m} (BG ^{\vee}) \to
K ^{\mathbb{Z} _{2}} _{m} (R). 
\end{align*}
If $G$ is simply laced (e.g. $G=SU (n)$) it is not hard to see that the images
of these maps are the same. In this case $G/T = G ^{\vee}/T$
and then the
images of $G$ and $G ^{\vee}$ in $\operatorname {Ham} (G/T)
$ coincide, so that the
assertion readily follows. Motivated by this:
\begin{proposal} \label{con_} For each compact Lie group $ G$ the
images of $h _{G,m}$ and $h _{G ^{\vee},m}$ coincide.
\end{proposal}
Furthermore, there are ``lifts'' \footnote {Given Keller's
theorem outlined in Section \ref{sec_Relation with Toen's
secondary K-theory}, they may be lifts in the formal sense,
that is the Hochschild functor
should induce a map of spectra $\mathfrak h$, such that
$\phi _{\mathcal{L},G} \circ \mathfrak h = h _{G}$.
However we do not verify this.} as in
\eqref{equation_phiL}:
\begin{align*} 
\phi  _{\mathcal{L}, G}: \pi _{m} (BG) \to K ^{Cat}
_{m} (R), \\
\phi _{\mathcal{L}, G ^{\vee}}: \pi _{m} (BG ^{\vee}) \to
K ^{Cat} _{m} (R),
\end{align*}
and we can again ask if the images coincide. For $ m=0$,
this is saying that $ \operatorname {Fuk}  (G/T)$ and $\operatorname {Fuk} (G
^{\vee}/T) $ \footnote {Or to be more restrictive their
models $C _{\mathcal{L}}$ as in Section
\ref{sec_Generalizing to GT}.} represent the same classes in $K ^{Cat} _{0}
(R)$. This assertion holds if these categories were
Morita equivalent, but it is much weaker. The case $m=0$ is
then loosely comparable to the homological mirror symmetry
conjecture of Kontsevich. 

Note that the domains of $\phi
_{\mathcal{L}, G}$ and $\phi 
_{\mathcal{L}, G ^{\vee}}$ and the manifolds $G/T$ and $G
^{\vee}/T$ are generally very different. The idea is that the images in (categorified) algebraic $K$-theory 
depend only on some representation theoretic data that is
invariant under Langlands duality. Thematically, such a phenomenon already
appears in electric-magnetic dualities, and this is one
inspiration. A related work on topological aspects
of Langlands duality/correspondence is 
\cite{cite_topologicalLanglands}. The seminal reference for
electric-magnetic duality vis-à-vis Langlands duality is  
Kapustin-Witten \cite{cite_KapustinWitten}. 
\section{Acknowledgments} This work was initiated during my
stay as a member at IAS in Spring 2024, and I am very
grateful for all that was provided. Many thanks to Bertrand
To\"en, Jacob Lurie, Nick Sheridan and Chi Hong Chow for
answering many background questions.   

\bibliographystyle{abbrvurl} 
\bibliography{link.bib}
\end{document}